\documentclass[12pt]{amsart}

\usepackage{mathpple}      
\usepackage{charter}        

\newtheorem{theorem}{Theorem}[section]

\newtheorem{proposition}[theorem]{Proposition}
\newtheorem{corollary}[theorem]{Corollary}

\theoremstyle{definition}
\newtheorem{definition}[theorem]{Definition}
\newtheorem{example}[theorem]{Example}
\newtheorem{conjecture}[theorem]{Conjecture}

\theoremstyle{remark}
\newtheorem{remark}[theorem]{Remark}

\numberwithin{equation}{section}

\def\RR{\mathbb{R}}

\def\seq{\subseteq}

\begin{document}

\title{On the strong non-rigidity of certain tight
Euclidean designs}
\author{Eiichi Bannai, Etsuko Bannai,}
\address{Eiichi Bannai, Etsuko Bannai,
and Djoko Suprijanto\\
Graduate School of Mathematics\\
Kyushu University\\
Hakozaki 6-10-1, Higashi-ku, Fukuoka 812-8581\\
JAPAN.}
\author{Djoko Suprijanto}
\address{Djoko Suprijanto\\
On leave from Department of Mathematics\\
Bandung Institute of Technology\\
Jalan Ganesha 10, Bandung 40133\\
INDONESIA.}
\thanks{The third author is supported in part by
the Monbukagakusho scholarship.}

\date{}
\maketitle

\begin{abstract}
\noindent
We study the non-rigidity of Euclidean $t$-designs,
namely we study when Euclidean designs
(in particular certain tight Euclidean designs) can be
deformed keeping the property of being Euclidean
$t$-designs. We show that certain tight Euclidean
$t$-designs are non-rigid, and in fact satisfy a stronger
form of non-rigidity which we call strong non-rigidity.
This shows that there are plenty of non-isomorphic
tight Euclidean $t$-designs for certain parameters,
which seems to have been unnoticed before. We also
include the complete classification of tight Euclidean
$2$-designs.
\end{abstract}

\bigskip
\noindent
\section{Introduction}
\par
The concept of spherical design was introduced by
Delsarte, Goethals and Seidel \cite{D-G-S}  in 1977
for finite sets in the unit sphere $S^{n-1}$
(in the Euclidean space $\RR^n$) . It measures how
much the finite set approximates the sphere $S^{n-1}$
with respect to the  integral of polynomial functions.
The exact definition is given as follows.

\noindent
\begin{definition}\label{Def1.1}
Let $t$ be a positive integer.  A finite nonempty subset
$X \seq S^{n-1}$ is called a \emph{spherical $t$-design}
if the following condition holds:
\begin{equation} \label{Eq1.1}
\frac{1}{|S^{n-1}|}\int_{S^{n-1}}f(\boldsymbol x)
d\sigma(\boldsymbol x)=\frac{1}{|X|}
\sum_{\boldsymbol x \in X}
f(\boldsymbol x),
\end{equation}
for any polynomial $f(\boldsymbol x)
\in \RR[x_1,x_2,\ldots,x_n]$ of degree at most $t$,
where
$\sigma(\boldsymbol x)$ is the $O(n)$-invariant
measure on $S^{n-1}$ and $|S^{n-1}|$ is the area
of the sphere $S^{n-1}$.
\end{definition}

The concept of spherical $t$-design was generalized
by Neumaier and Seidel \cite{N-S} in the following
two ways: (i) to drop the condition that it is on a sphere,
(ii) to allow weight. The new concept is called
\emph{Euclidean $t$-design}.  This concept is closely
related to the cubature formulae in numerical analysis
and approximation theory (see, e.g., \cite{D-X}),  and a similar concept such
as rotatable design has already existed also in
mathematical statistics (see, e.g., \cite{B-H,K}.)

Recently, the first and second authors
\cite{B-B-2}, slightly modified the Neumaier and
Seidel's definition of Euclidean $t$-design by dropping
the assumption of excluding the origin.
We will review the definition below.

Let $X$ be a finite set in $\RR^n,~n \geq 2.$
Let $\{r_1,~r_2,~\ldots,~r_p \}=\{\| \boldsymbol x \|,
~\boldsymbol x \in X \},$ where $\| \boldsymbol x \|$
is a norm of $\boldsymbol x$ defined by standard
inner product in $\RR^n$ and $r_i$ is possibly $0$.
For each $i$, we define $S_i=\{\boldsymbol x \in
\RR^n,~\| \boldsymbol x \|=r_i \}$, the sphere of radius
$r_i$ centered at 0.  We say that $X$ is supported by
the $p$ concentric spheres $S_1,~S_2,~\ldots,~S_p$.
If $r_i=0$, then $S_i=\{0\}$. Let $X_i=X \cap S_i,
~\mbox{for }1 \leq i \leq p.$  Let
$\sigma(\boldsymbol x)$
be the $O(n)$-invariant measure on the unit sphere
$S^{n-1} \seq \RR^n.$  We consider the measure
$\sigma_i(\boldsymbol x)$ on each $S_i$ so that
$|S_i|=r_i^{n-1}|S^{n-1}|$, with $|S_i|$ is the
surface area of $S_i$.  We associate a positive real
valued function $w$ on $X$, which is called a
\emph{weight} of $X$.  We define $\displaystyle w(X_i)
=\sum_{\boldsymbol x \in X_i}w(\boldsymbol x)$.
Here if $r_i=0,$ then we define $\displaystyle
\frac{1}{|S_i|}\int_{S_i}f(\boldsymbol x)
d \sigma_i(\boldsymbol x)=f(0),$ for any function
$f(\boldsymbol x)$ defined on $\RR^n$.  Let $\displaystyle S=\bigcup_{i=1}^p S_i.$
Let $\varepsilon_S \in \{0,~1\}$ be  defined by
\[
\varepsilon_S=\left\{
                    \begin{array}{cc}
                     1, & 0 \in S\\
                     0, & 0 \notin S
                    \end{array}.
              \right.
\]
We give some more notation we use. Let $\mbox{Pol}(\RR^n)=\RR[x_1,x_2,\ldots, x_n]$
be the vector space of polynomials in $n$ variables
$x_1,~x_2, \ldots, ~x_n.$  Let $\mbox{Hom}_l(\RR^n)$
be the subspace of $\mbox{Pol}(\RR^n)$ spanned by
homogeneous polynomials of degree $l$. Let
$\mbox{Harm}(\RR^n)$ be the subspace of
$\mbox{Pol}(\RR^n)$ consisting of all harmonic
polynomials.  Let $\mbox{Harm}_l(\RR^n)
= \mbox{Harm}(\RR^n) \cap \mbox{Hom}_l(\RR^n).$
Then we have $\displaystyle \mbox{Pol}_l(\RR^n)
= \oplus_{i=0}^l \mbox{Hom}_i(\RR^n)$.
Let $\displaystyle \mbox{Pol}_l^*(\RR^n)
= \oplus_{i\equiv l(2), \atop  0 \leq i \leq l}
\mbox{Hom}_i(\RR^n).$  Let $\mbox{Pol}(S)$,
$\mbox{Pol}_l(S)$, $\mbox{Hom}_l(S)$,
$\mbox{Harm}(S)$, $\mbox{Harm}_l(S)$ and
$\mbox{Pol}_l^*(S)$ be the sets of corresponding
polynomials restricted to the union $S$ of $p$
concentric spheres.  For example
$\mbox{Pol}(S)=\{f|_S,~f \in \mbox{Pol}(\RR^n)\}$.

With the notation mentioned above, we define a
Euclidean $t$-design as follows.
\noindent
\begin{definition}\label{Def1.2}
Let $X$ be a finite set with a weight function $w$
and let $t$ be a positive integer.  Then $(X,w)$ is
called a \emph{Euclidean $t$-design} in $\RR^n$ if
the following condition holds:
\[
\sum_{i=1}^p \frac{w(X_i)}{|S_i|}\int_{S_i}
f(\boldsymbol x)d\sigma_i(\boldsymbol x)
=\sum_{\boldsymbol x \in X}w(\boldsymbol x)
f(\boldsymbol x),
\]
for any polynomial $f(\boldsymbol x) \in
\mbox{Pol}(\RR^n)$ of degree at most $t$.
\end{definition}

Let $X$ be a Euclidean $2e$-design in $\RR^n.$
Then it is known that $|X| \geq \dim(\mbox{Pol}_e(S)).$
Let $X$ be an antipodal $(2e+1)$-design in $\RR^n.$
Then it is also known that $|X^*| \geq
\dim(\mbox{Pol}_e^*(S)).$  Here $X^*$ is an antipodal
half part of $X$ satisfying $X^* \cup (-X^*)=X$ and $X^*
\cap (-X^*)=\{0\}$ or $\emptyset$.  Although better lower
bounds are proved in \cite{D-S} and \cite{N-S},
$\dim(\mbox{Pol}_e(S))$ and $\dim(\mbox{Pol}_e^*(S))$
are considered to be very natural.  We define the
following tightness for the Euclidean designs
(c.f. \cite{B-B-2, Et}).

\begin{definition}\label{Def1.4}
Let $X$ be a Euclidean $2e$-design supported by $S$.
If $|X|=\dim(\mbox{Pol}_e(S))$ holds we call \emph{$X$ a
tight $2e$-design on $S$}.  Moreover if $\dim(\mbox{Pol}_e(S))=\dim(\mbox{Pol}_e(\RR^n))$ holds,
then $X$ is called \emph{a tight Euclidean $2e$-design.}
\end{definition}

\begin{definition}\label{Def1.5}
Let $X$ be an antipodal Euclidean $(2e+1)$-design
supported by $S$.  Assume $w(\boldsymbol x)
=w(-\boldsymbol x)$ for any $\boldsymbol x \in X$.
If $|X^*|=\dim(\mbox{Pol}_e^*(S))$ holds, we call
\emph{$X$ an antipodal tight $(2e+1)$-design on $S$}.
Moreover if $\dim(\mbox{Pol}_e^*(S))
=\dim(\mbox{Pol}_e^*(\RR^n))$ holds,
then $X$ is called  \emph{an antipodal tight
Euclidean $(2e+1)$-design.}
\end{definition}

In Section 2, we give some more basic facts about the
Euclidean designs. In Section 3, we give the definition
of the strong non-rigidity of Euclidean designs.  Our
main theorem is Theorem \ref{Theo3.8}, in which we
show that the following known examples of tight
Euclidean designs are strongly non-rigid:  tight
Euclidean $4$-designs in $\RR^2,$  tight Euclidean
$2$-designs in $\RR^n$ supported by one sphere, or
equivalently, tight spherical $2$-designs.  We also show
that antipodal tight spherical $3$-designs in $\RR^2$ in
the sense of Euclidean design as well as antipodal tight
Euclidean $5$-designs in $\RR^2$ are strongly non-rigid.
The implication of these facts are the existence of
infinitely many non-isomorphic tight Euclidean designs
with the given strength. This is quite contrary to the case
of spherical designs, where tight spherical $t$-designs
are rigid, and
there are only finitely many tight spherical $t$-designs
in $S^{n-1}$ (up to orthogonal transformations) for each
fixed pair of $n$ and $t$.

The complete classification of tight Euclidean
$2$-designs in $\RR^n$ is given in Section 4.
We also show that any finite subset $X \seq \RR^n$
of cardinality $n+1$ is a Euclidean $2$-design if and
only if $X$ is a $1$-inner product set with negative
inner product value.  Here we say $X \in \RR^n$ is
an $e$-inner product set if
$|\{\langle \boldsymbol x,\boldsymbol y \rangle,
~\boldsymbol x,\boldsymbol y \in X,\boldsymbol x
\neq \boldsymbol y \}|=e$ holds.  We remark that
$\displaystyle |X| \leq \dim(\mbox{Pol}_e(\RR^n))
=\left(n+e \atop e \right)$ holds
for any $e$-inner product set $X$ in $\RR^n$
(\cite{D-F}).

\section{Basic facts on Euclidean designs}
The following theorem gives a condition which is
equivalent to the definition of Euclidean $t$-designs.

\noindent
\begin{theorem}[Neumaier-Seidel]\label{TheoN-S}
Let $X$ be a finite nonempty subset in $\RR^n$
with weight function $w$.  Then the following (1) and (2)
are equivalent:
\begin{itemize}
\item[(1)] $X$ is a Euclidean $t$-design.
\item[(2)] $\displaystyle \sum_{\boldsymbol u \in X}
w(\boldsymbol u)
\|\boldsymbol u\|^{2j}
\varphi(\boldsymbol u)=0,$ for any polynomial $\varphi \in \emph{Harm}_l(\mathbb{R}^n)$\\
with $1 \leq l \leq t$ and $0 \leq j \leq \lfloor \frac{t-l}{2} \rfloor$.
\end{itemize}
\end{theorem}

We will use the condition (2) of Theorem \ref{TheoN-S}
in what follows.  Theorem \ref{TheoN-S} implies the
following proposition.

\begin{proposition}[\cite{B-B-2}, Proposition 2.4]
\label{Prop2.2}
Let $(X,w)$ be a Euclidean $t$-design in $\RR^n$.
Then the following (1) and (2) hold:
\begin{itemize}
\item[(1)] Let $\lambda$ be a positive real number and $X^{\prime}=\{\lambda \boldsymbol u,~\boldsymbol u \in X\}.$  Then $X^{\prime}$ is also a Euclidean $t$-design with weight $w^{\prime}$ defined by $w^{\prime}=w(\frac{1}{\lambda} \boldsymbol u^{\prime}),
~\boldsymbol u^{\prime} \in X^{\prime}.$
\item[(2)] Let $\mu$ be a positive real number and $w^{\prime}(\boldsymbol u)=\mu w(\boldsymbol u)$ for any $\boldsymbol u \in X$.  Then X is also a Euclidean $t$-design with respect to the weight $w^{\prime}$.
\end{itemize}
\end{proposition}

\begin{remark}
The concept of spherical designs can be obviously generalized to a finite subset in a sphere with arbitrary radius $r.$ For this purpose, we just need to replace the unit sphere $S^{n-1}$ by $S^{n-1}(r)$, the sphere of radius $r$, in the formula (\ref{Eq1.1}) in Definition \ref{Def1.1}. Therefore, we regard the spherical designs as Euclidean designs with $p=1$ and constant weight function $w(x).$\end{remark}

We also need the proposition below in the subsequent sections.

\begin{proposition}\label{Prop2.4}
Let $(X,w)$ be a tight $2e$-design or antipodal tight $(2e+1)$-design on a union of
concentric spheres $S$ in $\RR^n.$  Then the weight function $w$ is constant on each sphere.
\end{proposition}

\begin{proof}
See \cite{B-B-2} for $2e$-design case  and \cite{Et} for $(2e+1)$-design case.
\end{proof}

Let $(X,w)$ be a finite weighted subset in $\RR^n$.  Let $S_1,S_2,\ldots,S_p$ be the $p$ concentric spheres supporting $X$ and let $\displaystyle S= \bigcup_{i=1}^p S_i$.

For any $\varphi,~\psi \in \mbox{Harm}(\RR^n),$ we define the following inner-product
\[
\langle \varphi, \psi \rangle = \frac{1}{|S^{n-1}|} \int_{S^{n-1}}
\varphi(\boldsymbol x)
\psi(\boldsymbol x)
d\sigma(\boldsymbol x).
\]

Then the following properties are well known (see \cite{Erdel, D-G-S, D-S, B-B-1, Et}).

\begin{proposition}\label{Prop2.5}
(1) $\emph{Harm}(\RR^n)$ is a positive definite inner-product space under $\langle -,-\rangle$ and has the orthogonal decomposition $\displaystyle \emph{Harm}(\RR^n)= \perp _{i=0}^\infty \emph{Harm}_i(\RR^n).$\\
(2) $\displaystyle \emph{Pol}_e(\RR^n)=\bigoplus _{0 \leq i+2j \leq e}
\|\boldsymbol x\|^{2j}\emph{Harm}_i(\RR^n) \mbox{ with }\dim(\emph{Pol}_e(\RR^n))=\left( n+e \atop e  \right).$ \\
(3) $\displaystyle \emph{Pol}_e(S) = \left\langle \| \boldsymbol x \|^{2j}~|~ 0 \leq j \leq ~\min\left\{p-1,\left[\frac{e}{2}\right]\right\} \right\rangle \oplus
\left\{\bigoplus _{
1 \leq i \leq e, \atop
0 \leq j \leq \min \{p-\varepsilon_S-1, \left[\frac{e-i}{2} \right] \}
} \| \boldsymbol x \|^{2j} \emph{Harm}_i(S) \right\}$.\\
Furthermore if $p \leq [\frac{e+\varepsilon_S}{2}],$ then
$$
\dim(\emph{Pol}_e(S)) = \varepsilon_S + \sum_{i=0}^{2(p-\varepsilon_S)-1} \left(\begin{array}{c} n+e-i-1\\ n-1 \end{array} \right),
$$
and if $p \geq [\frac{e + \varepsilon_S}{2}]+1$, then
$$
\dim(\emph{Pol}_e(S))=\left(\begin{array}{c} n+e \\ e \end{array} \right),
$$
where $e$ is a non-negative integer.\\
(4) $\displaystyle \emph{Pol}_e^*(\RR^n)=\bigoplus_{i=0}^{[\frac{e}{2}]}\bigoplus _{j=0}^{[\frac{e-2i}{2}]}
\|\boldsymbol x\|^{2j}\emph{Harm}_{e-2i-2j}(\RR^n) \mbox{ with }\dim(\emph{Pol}_e^*(\RR^n))=\sum_{i=0}^{[\frac{e}{2}]}\left( n+e-2i-1 \atop n-1  \right).$ \\
(5) $\displaystyle \emph{Pol}_e^*(S)=\bigoplus_{i=0}^{[\frac{e}{2}]}\bigoplus _{j=0}^{[\frac{e-2i}{2}]}
\|\boldsymbol x\|^{2j}\emph{Harm}_{e-2i-2j}(S)$.  Furthermore if $p \leq [\frac{e}{2}]$, then
\[ \displaystyle
\left\{
\begin{array}{ll}
\displaystyle \dim(\emph{Pol}_e^*(S)) = \sum_{i=0}^{p-1} \left(\begin{array}{c} n+e-2i-1\\ n-1 \end{array} \right) < \dim(\emph{Pol}_e^*(\RR^n)), & e ~\mbox{is odd or even and } 0 \notin S,\\
\displaystyle \dim(\emph{Pol}_e^*(S)) = 1 + \sum_{i=0}^{p-2} \left(\begin{array}{c} n+e-2i-1\\ n-1 \end{array} \right) < \dim(\emph{Pol}_e^*(\RR^n)), & e~\mbox{is even and } 0 \in S,
\end{array}
\right.
\]
and if $p \geq [\frac{e}{2}]+1$, then
\[
\dim(\emph{Pol}_e^*(S)) = \sum_{i=0}^{[\frac{e}{2}]} \left(\begin{array}{c} n+e-2i-1\\ n-1 \end{array} \right) = \dim(\emph{Pol}_e^*(\RR^n)).
\]
\end{proposition}

Let $h_l=\dim(\mbox{Harm}_l(\RR^n))$ and $\varphi_{l,1}, \ldots, \varphi_{l,h_l}$ be an orthonormal basis of $\mbox{Harm}_l(\RR^n)$ with respect to the inner-product defined above.  Then, by Proposition \ref{Prop2.5},

\[
\begin{array}{l}
\displaystyle \left\{
\|\boldsymbol x \|^{2j},~0 \leq j \leq \min\left\{p-1,\left[\frac{e}{2} \right]\right \} \right\} \bigcup \\
\displaystyle  \left\{
\|\boldsymbol x\|^{2j}
\varphi_{l,i}(\boldsymbol x),~1 \leq l
\leq e,~1 \leq i \leq h_l,~0 \leq j \leq \min \left\{p-\varepsilon_S-1,\left[\frac{e-l}{2} \right]\right \} \right\}
\end{array}
\]
gives a basis of $\mbox{Pol}_e(S)$.

Now, we are going to construct a more convenient basis of $\mbox{Pol}_e(S)$ for our purpose.  Let $\mathcal{G}(\RR^n)$ be the subspace of $\mbox{Pol}_e(S)$ spanned by $\{
\|\boldsymbol x\|^{2j},~0 \leq j \leq p-1 \}$.  Let $\mathcal{G}(X)=\{g|_X,~g \in \mathcal{G}(\RR^n) \}$.  Then $\{\| \boldsymbol x \|^{2j},~ 0 \leq j \leq p-1 \}$ is a basis of $\mathcal{G}(X)$.  We define an inner-product $\langle -,-\rangle_l$ on $\mathcal{G}(X)$ by

\begin{equation}\label{Eq2.1}
\langle f, g \rangle_l =
\sum_{\boldsymbol x \in X}
w(\boldsymbol x)
\| \boldsymbol x \|^{2l}f(\boldsymbol x)
g(\boldsymbol x),\mbox{ for }1 \leq l \leq e.
\end{equation}

We apply the Gram-Schmidt method  to the basis $\{\|\boldsymbol x\|^{2j},~0 \leq j \leq p-1 \}$ to construct an orthonormal basis
\[
\{g_{l,0}(\boldsymbol x),~
g_{l,1}(\boldsymbol x),~\ldots,~
g_{l,p-1}(\boldsymbol x) \}
\]
of $\mathcal{G}(X)$ with respect to the inner-product $\langle-,-\rangle_l$.  We can construct them so that for any $l$ the following holds:
\[
\begin{array}{l}
g_{l,j}(\boldsymbol x) \mbox{ is a linear combination of }1,~\|\boldsymbol x\|^2,~\ldots,~\|\boldsymbol x\|^{2j}, \mbox{ with }\deg(g_{l,j})=2j,\\
\mbox{for } 0 \leq j \leq p-1.
\end{array}
\]
For example, we can express
$g_{l,0}(\boldsymbol x)$ as
\begin{equation}\label{Eq2.2}
g_{l,0}(\boldsymbol x) \equiv \frac{1}{\sqrt{a_l}}, \mbox{ with } a_l
=\sum_{\boldsymbol x \in X}
w(\boldsymbol x) \| \boldsymbol x \|^{2l}.
\end{equation}
Now we are ready to give a new basis for $\mbox{Pol}_e(S).$  Let us consider the following sets:
\[
\begin{array}{l}
\displaystyle \mathcal{H}_0=\left\{g_{0,j}~\bigm|~ 0 \leq j \leq \min\left\{p-1,\left[\frac{e}{2}\right]\right\} \right\},\\
\displaystyle \mathcal{H}_l=\left\{g_{l,j}\varphi_{l,i}~\bigm|~0 \leq j \leq \min\left \{p-\varepsilon_S-1, \left[\frac{e-l}{2}\right]\right \},~1 \leq i \leq h_l \right\}, \mbox{ for } 1 \leq l \leq e.
\end{array}
\]
Then $\displaystyle \mathcal{H}=\bigcup_{l=0}^e \mathcal{H}_l$ is a basis of $\mbox{Pol}_e(S)$.

\begin{proposition}\label{Prop2.6} If $(X,w)$ is a tight $2e$-design on $S$, then the following (1) and (2) hold:
\begin{itemize}
\item[(1)] The weight function of $X$ satisfies
\begin{equation}\label{Eq2.3}
\sum_{1 \leq l \leq e, \atop 0 \leq j \leq \min\{p-\varepsilon_S-1, [\frac{e-l}{2}] \}} \| \boldsymbol u \|^{2l} g_{l,j}^2(\boldsymbol u) Q_l(1) + \sum _{j=0}^{\min \{p-1,[\frac{e}{2}]\}} g_{0,j}^2(\boldsymbol u) = \frac{1}{w(\boldsymbol u)}, \mbox{ for all } \boldsymbol u \in X.
\end{equation}
\item[(2)] For any distinct points $\boldsymbol u,~\boldsymbol v \in X,$ we have
\begin{equation}\label{Eq2.4}
\sum_{1 \leq l \leq e, \atop 0 \leq j \leq \min\{p-\varepsilon_S-1, [\frac{e-l}{2}] \}} \| \boldsymbol u \|^l
\| \boldsymbol v \|^l g_{l,j}
(\boldsymbol u) g_{l,j}
(\boldsymbol v) Q_l\left(\frac{\langle \boldsymbol u,\boldsymbol v \rangle}{
\|\boldsymbol u\|
\|\boldsymbol v\|}\right) + \sum_{j=0}^{\min\{p-1,[\frac{e}{2}]\}} g_{0,j}(\boldsymbol u) g_{0,j}(\boldsymbol v) = 0.
\end{equation}
\end{itemize}
Here $\langle \boldsymbol u,
\boldsymbol v \rangle$ is the standard inner-product in Euclidean space $\RR^n$ and $Q_l(\alpha)$ is the Gegenbauer polynomial of degree $l$. Moreover, for the case $e=1$ the converse is also true, namely, if (1) and (2) hold, then $X$ is a tight $2e$-design on $S$.
\end{proposition}

\begin{proof}(\emph{c.f.} \cite{B-B-2})  Let $X$ be a tight $2e$-design on $S$.  Let $M$ be the matrix indexed by $X \times \mathcal{H}$ which is defined by
\[
M(\boldsymbol u,~g_{l,j}\varphi_{l,i})=\sqrt{w(\boldsymbol u)}g_{l,j}(\boldsymbol u)\varphi_{l,i}
(\boldsymbol u).
\]
Then $^tMM=I.$  Furthermore, since $X$ is tight, then $M$ is square, and hence $M~^tM=I$.  Therefore, for nonzero vectors $\boldsymbol u,\boldsymbol v \in X,$ we have
\begin{align*}
\frac{M~^tM(\boldsymbol u,
\boldsymbol v)}{\sqrt{w(\boldsymbol u)
w(\boldsymbol v)}} & =  \sum_{1 \leq l \leq e, \atop 0 \leq j \leq \min\{p-\varepsilon_S-1,[\frac{e-l}{2}]\}}
\|\boldsymbol u\|^l \|\boldsymbol v\|^l
g_{l,j}(\boldsymbol u)g_{l,j}
(\boldsymbol v)Q_l\left(
\frac{\langle \boldsymbol u,
\boldsymbol v \rangle}
{\|\boldsymbol u\| \|\boldsymbol v\|}\right)\\
&  + \sum_{j=0}^{\min\{p-1,[\frac{e}{2}]\}}g_{0,j}(\boldsymbol u)g_{0,j}(\boldsymbol v).
\end{align*}
Hence, by restricting $\boldsymbol u
=\boldsymbol v$ and $\boldsymbol u \neq \boldsymbol v$, we have (1) and (2), respectively.

Now, we prove that if $e=1$ the converse holds.  By Theorem \ref{TheoN-S}, it is enough to show

\[
\sum_{\boldsymbol u \in X}
w(\boldsymbol u)\varphi(\boldsymbol u)=0, \mbox{ for any } \varphi \in \mbox{Harm}_l(\RR^n), \mbox{ with } l=1,2.
\]

It is known that $\{\varphi_{1,i}(\boldsymbol x)=cx_i,~1 \leq i \leq n \}$ forms an orthonormal basis of $\mbox{Harm}_1(\RR^n)$, where $\displaystyle c=\sqrt{\frac{|S^{n-1}|}{\int_{S^{n-1}}x_i^2 d\sigma(\boldsymbol x)}}.$  (Note that the constant $c$ is independent of the choice of the index $i$).  Then $\mbox{Harm}_2(\RR^n)$ is spanned by $\{\varphi_{1,i}\varphi_{1,j},~\varphi_{1,i}^2-\varphi_{1,j}^2,~1 \leq i < j \leq n \}$.

From $^tMM=I$, we have

\[
\sum_{\boldsymbol u \in X}
w(\boldsymbol u)g_{l,0}(\boldsymbol u)\varphi_{l,i}(\boldsymbol u)g_{l^\prime,0}(\boldsymbol u)\varphi_{l^\prime,
i^\prime}(\boldsymbol u)=\delta(i,i^\prime)\delta(l,l^\prime).
\]

For $l=1$ and $l^\prime=0,$ we have

\[
\sum_{\boldsymbol u \in X}
w(\boldsymbol u)\varphi_{1,i}(\boldsymbol u)=0, \mbox{ for any } \varphi_{1,i} \in \mbox{Harm}_1(\RR^n),
\]

while for $l=1=l^\prime$, we have

\[
\sum_{\boldsymbol u \in X}
w(\boldsymbol u)
\varphi_{1,i}(\boldsymbol u)
\varphi_{1,i^\prime}(\boldsymbol u)=\delta(i,i^\prime)\frac{1}{g_{1,0}^2}, \mbox{ for any } \varphi_{1,i},~\varphi_{1,i^\prime} \in \mbox{Harm}_1(\RR^n),
\]
since $g_{l,0}$ is a non-zero constant function on $X$.  Therefore we obtain
\[
\sum_{\boldsymbol u \in X}
w(\boldsymbol u)\varphi_{1,i}(\boldsymbol u)\varphi_{1,i^\prime}(\boldsymbol u)=0,
\]
and
\[
\sum_{\boldsymbol u \in X}w(\boldsymbol u)(\varphi_{1,i}^2(\boldsymbol u)-\varphi_{1,i^\prime}^2(\boldsymbol u))=0,
\]
for any $i \neq i^\prime$.  The result follows  from the fact that $\mbox{Harm}_2(\RR^n) = \langle \varphi_{1,i}\varphi_{1,j},~\varphi_{1,i}^2-\varphi_{1,j}^2,~1 \leq i < j \leq n \ \rangle$.
\end{proof}

\section{Rigidity of spherical and Euclidean designs}

We call a spherical $t$-design {\it non-rigid} (resp. {\it rigid}) if it cannot be (resp.  can be) deformed locally keeping the property that it is a spherical $t$-design. The exact definition is given as follows (c.f. \cite{Ban}).

\noindent
\begin{definition}\label{Def3.1}
A spherical $t$-design
$X=\{\boldsymbol x_i,~1 \leq i \leq N\} \seq S^{n-1}$ is called \emph{non-rigid} or \emph{deformable} in $\RR^n$ if for any $\varepsilon > 0$ there exists another spherical $t$-design $X^{\prime}=\{\boldsymbol x_i^{\prime},~1 \leq i \leq N\} \seq S^{n-1}$ such that the following two conditions hold:
\begin{itemize}
\item[(1)] $\|\boldsymbol x_i-
\boldsymbol x_i^{\prime} \|< \varepsilon,$  for $1 \leq i \leq N;$ and
\item[(2)] there is no any transformation $g \in O(n),$ with $g(\boldsymbol x_i)=\boldsymbol x_i^{\prime},$ for $1 \leq i \leq N.$
\end{itemize}
\end{definition}

Motivated by the above definition and Proposition \ref{Prop2.2}, we define a similar concept of rigidity and non-rigidity for Euclidean $t$-design, depending upon whether the designs can be transformed to each other by orthogonal transformations, scaling, or adjustment of the weight functions.  In the definition below, $O^*(n)=\langle O(n),g_{\lambda}, g^{\mu} \rangle$ denotes a group generated by an orthogonal group $O(n)$, a scaling $g_{\lambda}$ of $X:$
\[
g_{\lambda}:\left\{
\begin{array}{ccc}
(X,w)& \longrightarrow & (X^{\prime},w^{\prime})\\
\boldsymbol x & \longmapsto & \boldsymbol x^{\prime}
=\lambda \boldsymbol x\\
w^{\prime}(\boldsymbol x^{\prime})
=w(\boldsymbol x)
\end{array}
\right.,
\]
and an adjustment $g^{\mu}$ of weight function $w:$
\[
g^{\mu}:\left\{
\begin{array}{ccc}
(X,w)& \longrightarrow & (X^{\prime},w^{\prime})\\
\boldsymbol x & \longmapsto &
\boldsymbol x^{\prime}:
=\boldsymbol x\\
w^{\prime}(\boldsymbol x^{\prime})
=\mu w(\boldsymbol x)
\end{array}
\right..
\]

\noindent
\begin{definition}\label{Def3.2}
A Euclidean $t$-design $X=\left(
\{\boldsymbol x_i\}_{i=1}^{N},w \right) \seq \RR^n$ is called \emph{non-rigid} or \emph{deformable} in $\RR^n$ if for any $\varepsilon > 0$ there exists another Euclidean $t$-design $X^{\prime}=\left(\{\boldsymbol x_i^{\prime}\}_{i=1}^{N},w^{\prime} \right) \seq \RR^n$ such that the following two conditions hold:\begin{itemize}
\item[(1)] $\|\boldsymbol x_i-\boldsymbol x_i^{\prime} \|< \varepsilon,$ and $|w(\boldsymbol x_i)- w^{\prime}(\boldsymbol x_i^{\prime})| < \varepsilon,$ for $1 \leq i \leq N;$ and
\item[(2)] there is no any transformation $g \in O^*(n),$ with $g(\boldsymbol x_i)=\boldsymbol x_i^{\prime}$ for $1 \leq i \leq N.$
\end{itemize}
\end{definition}

It is well known that any tight spherical $t$-design is rigid, because the possible distances of any two points in the design are finitely many in number and determined by only $n$ and $t$ (see Theorem 5.11 and 5.12 in \cite{D-G-S} also see \cite{B-M-V} for the current status of the  classification of tight spherical $t$-designs).  A natural question is whether tight spherical $t$-designs are rigid as Euclidean $t$-designs.  We have the proposition below.

\noindent
\begin{proposition} \label{Prop3.6}
Any tight spherical $2e$-design is rigid as a Euclidean design, for $e\geq 2.$
\end{proposition}

\begin{proof}
Let $X$ be a tight spherical $2e$-design, with $e \geq 2.$ Suppose that $X$ is non-rigid as a Euclidean design, and $X$ can be deformed to $X^{\prime}.$  Since $X$ is rigid as a spherical design, $X^\prime$ cannot sit on one sphere.  It means that  $X^\prime$ should be a Euclidean $2e$-design supported by at least two spheres.  This implies that the size of $X^{\prime}$ becomes greater than the initial size of $X$,i.e., $\displaystyle |X^\prime| \geq \left(n+e \atop e \right) > |X|$, which is impossible.
\end{proof}

On the other hand, as we will show later, any tight spherical $2$- and $3$-design are non-rigid as Euclidean designs.

Now, let us consider the following two examples of tight Euclidean $4$-designs in $\RR^2$ given by Bannai and Bannai \cite{B-B-2} and also antipodal tight Euclidean $5$-designs in $\RR^2$ given in Bannai \cite{Et}.

\begin{example}[see \cite{B-B-2}]\label{Ex3.4}
Let $X(r)=X_1 \cup X_2(r)$, where $X_1=\left\{\left(1,0\right),\left(-\frac{1}{2},\frac{\sqrt{3}}{2}\right),\left(-\frac{1}{2},-\frac{\sqrt{3}}{2}\right) \right\}$ and $X_2(r)=\left\{\left(-r,0 \right),\left(\frac{r}{2},\frac{\sqrt{3}}{2}r \right),\left(\frac{r}{2},-\frac{\sqrt{3}}{2}r \right)\right\}$.  Let $w(\boldsymbol x)=1$ for $\boldsymbol x \in X_1$ and $w(\boldsymbol x)=\frac{1}{r^3}$ for $\boldsymbol x \in X_2(r)$.  If $r \neq 1$, then $X(r)$ is a tight Euclidean $4$-design.
\end{example}

\begin{example}[see \cite{Et}]\label{Ex3.5}
Let $X(r)=X_1 \cup X_2(r)$ where $X_1=\{(\pm 1,0),(0,\pm 1)\}$ and $X_2=\left\{\left(\pm \frac{r}{\sqrt{2}},\pm \frac{r}{\sqrt{2}}\right)\right\}$. Let
$w(\boldsymbol x)=1$ for
$\boldsymbol x \in X_1$ and
$w(\boldsymbol x)
=\frac{1}{r^4}$ for $\boldsymbol x
\in X_2(r)$.  If $r \neq 1,$ then $X(r)$ is an antipodal tight Euclidean $5$-design.
\end{example}

In both examples above, we can easily see that if we move all the points on $X_2(r)$ simultaneously by changing the radius $r$ while the other points remain sitting on the original position, the resulting designs are again Euclidean designs of the same type.  This kind of transformation is not contained in the group $O^*(n)$ since $X(r)$ and $X(r^{\prime})$ are not similar to each other for any $r \neq r^{\prime}$.  Hence the designs are non-rigid.

In the deformation explained above, all points on the same sphere move to the new one.  One natural question is, what will happen if we deform $X$ so that some two points from the same sphere move to distinct two spheres?  This question bring us to the notion of strong non-rigidity, a special kind of non-rigidity.

\begin{definition}[strong non-rigidity]
Let $X=\left(\{\boldsymbol x_i\}_{i=1}^N,w \right)$ be a Euclidean $t$-design in $\RR^n$.  If $X$ satisfies the following condition we say $X$ is \emph{strongly non-rigid} in $\RR^n$:\\
For any $\varepsilon >0$ there exists a Euclidean $t$-design $X^{\prime}=\left(\{\boldsymbol x_i^{\prime}\}_{i=1}^N,w^{\prime} \right)$ such that the following two conditions hold:
\begin{itemize}
\item[(1)]  $\|\boldsymbol x_i-
\boldsymbol x_i^{\prime}\| <
\varepsilon$ and $|w(\boldsymbol x_i)-w^{\prime}(\boldsymbol x_i^{\prime})| < \varepsilon,$ for any $1 \leq i \leq N$; and
\item[(2)]  There exist distinct $i,j$ satisfying $\|\boldsymbol x_i\|=\|\boldsymbol x_j\|$ and
$\|\boldsymbol x_i^{\prime}\| \neq
\|\boldsymbol x_j^{\prime} \|$.
\end{itemize}
\end{definition}

\begin{remark}
It is clear that any strongly non-rigid Euclidean $t$-design is non-rigid, since the condition (2) above implies that the transformation:
\[
\boldsymbol x_i \longmapsto
\boldsymbol x_i^{\prime}, ~1 \leq i \leq N,
\]
is not contained in $O^*(n).$
\end{remark}

In the following we will prove the theorem below.

\begin{theorem}\label{Theo3.8}
The following tight Euclidean $t$-designs are strongly non-rigid:
\begin{itemize}
\item[(1)]  Tight spherical $2$-designs in $S^{n-1}$ considered as tight Euclidean $2$-designs.
\item[(2)] Antipodal tight spherical $3$-designs in $S^1$ considered as tight Euclidean $3$-designs.
\item[(3)] Tight Euclidean $4$-designs in $\RR^2$  supported by 2 concentric spheres.
\item[(4)] Antipodal tight Euclidean $5$-designs in $\RR^2$  supported by 2 concentric spheres.
\end{itemize}
\end{theorem}

Theorem \ref{Theo3.8} implies the following corollary.

\begin{corollary}\label{Col3.9}
There are infinitely many tight Euclidean designs of the following type:
\begin{itemize}
\item[(1)]  $2$-designs in $\RR^n$ supported by $p=2,3,\ldots,n+1$ concentric spheres, respectively.
\item[(2)]  Antipodal $3$-designs in $\RR^2$ supported by $2$ concentric spheres.
\item[(3)]  $4$-designs in $\RR^2$ supported by $3$ and $4$ concentric spheres.
\item[(4)]  Antipodal $5$-designs in $\RR^2$ supported by $3$ and $4$ concentric spheres.
\end{itemize}

\end{corollary}

Corollary \ref{Col3.9} says about the existence of quite plenty of tight Euclidean $t$-designs, contrary to the initial guess made by Neumaier and Seidel and also Delsarte and Seidel respectively in \cite{N-S} and \cite{D-S}.  We remark here that antipodal tight Euclidean $3$-designs in $\RR^n$ have been completely classified in \cite{Et}.

\section{Proof of Theorem \ref{Theo3.8}}
We will prove Theorem \ref{Theo3.8} using the implicit function theorem described below.

Let $X$ be a tight Euclidean $t$-design in $\RR^n.$  Let $|X|=N,~X
=\{\boldsymbol u_i,~1 \leq i \leq N \}$ and
$\boldsymbol u_i=(u_{i,1},u_{i,2},\ldots,u_{i,n})$ for $1 \leq i \leq N$.  Let $w_i$ be the weight of $\boldsymbol u_i$, for $1 \leq i \leq N$.  Then we consider $(u_{i,1},u_{i,2}, \ldots, u_{i,n},w_i,~1 \leq i \leq N)$ as a vector $\eta=(\eta_1,\eta_2,\ldots,\eta_{(n+1)N}) \in \RR^{(n+1)N}$ whose entries are given by $u_{i,1},u_{i,2},\ldots,u_{i,n},w(u_i),~\mbox{for }1 \leq i \leq N$.  Let $\xi=(\xi_1,\xi_2,\ldots,\xi_{(n+1)N}) \in \RR^{(n+1)N}$ be the vector variable whose entries are defined by $(x_{i,1},x_{i,2}, \ldots, x_{i,n},w_i),~1 \leq i \leq N).$  Then $\eta$ is a common zero point of a given set of polynomials $f_1(\xi),f_2(\xi),\ldots,f_K(\xi)$ in the vector variable $\xi$ (c.f. Theorem \ref{TheoN-S} (2)).  Let $I=\{i,~1 \leq i \leq (n+1)N \}$ and $I ^{\prime} \seq I$.  We denote by $J$ the Jacobian
of the system of equations and $J'$ be a
sub matrix of $J$ of size $K\times K$:
\[
J=\left(\frac{\partial f_i}{\partial \xi_k}\right) _{1\leq i\leq K,\atop k \in I},\quad
J^{\prime}=\left(\frac{\partial f_i}{\partial \xi_k}\right) _{1\leq i\leq K,\atop k \in I \backslash I^{\prime}}.
\]
Assume $|I\backslash I^{\prime}|=K$ and that $\mbox{rank}(J^{\prime})=K$ holds at $\eta$, i.e.,
$J$ is of full rank at $\eta$.  We may assume $I \backslash I^{\prime}=\{1,2,\ldots,K\}$ by reordering the components of the vectors $\xi$ and $\eta$.  Let $\xi^{\prime}=(\xi_i,~i \in I^{\prime})$ and $\eta^{\prime}=(\eta_i,~i \in I^{\prime}).$  Then the implicit function theorem tells us that there exist unique continuously differentiable function $\Psi(\xi^{\prime})=(\psi_i(\xi^{\prime}),~i \in I \backslash I^{\prime})$ satisfying the following conditions:
\begin{itemize}
\item[(1)]  For any $1 \leq j \leq K,$
\[
f_j(\psi_1(\xi^{\prime}),\psi_2(\xi^{\prime}),\ldots,\psi_K(\xi^{\prime}),\xi^{\prime})=0
\]
holds in some small neighborhood of $\eta^{\prime}$.
\item[(2)]  $\psi_i(\eta^{\prime})=\eta_i,$ for any $1 \leq i \leq K$.
\end{itemize}

Let $\xi_i=\psi_i(\xi^{\prime}),$ for $1 \leq i \leq K$.  Then for any $\xi^{\prime}$ in a small neighborhood of $\eta^{\prime},$ $X^{\prime}=\{\xi_i,~i \in I \}$ is a Euclidean $t$-design.  Since $\psi_i(\xi^{\prime}),~1 \leq i \leq K,$ are continuous function of $\xi^{\prime},$ we can make $|\xi_i-\eta_i|< \varepsilon$ for any given positive real number $\varepsilon.$  For example, if $X$ is a tight Euclidean $2e$-design and $I^{\prime}$ contains all the indices corresponding to the variables $w_1,w_2,\ldots,w_N,$ then we can make every point in $X^{\prime}$ having distinct weight values.  Since, by Proposition \ref{Prop2.4}, a tight Euclidean $2e$-design $X^{\prime}$ must have constant weight on each sphere which support $X^{\prime},$ every point of $X^{\prime}$ must be on the different spheres.

In the following we apply this method to the tight spherical $2$-designs on $S^{n-1}$, tight spherical $3$-designs on $S^1,$ tight Euclidean $4$-designs in $\RR^2$ (Example \ref{Ex3.4}) and antipodal tight Euclidean $5$-designs in $\RR^2$ (Example \ref{Ex3.5}).\\

\noindent
{\bf  (1) Tight spherical 2-designs on $S^{n-1}$.}\\
A tight spherical $2$-design on $S^{n-1}$ is tight as
Euclidean $2$-design since $\dim(\mbox{Pol}_1(S^1))
=\dim(\mbox{Pol}_1(\RR^2))$ holds.
In the following section we will give the classification of all the
Euclidean tight $2$-designs in $\mathbb R^n$. However, since the
concept of rigidity or strong non rigidity
is very important, and also we would like to show how we applied
the implicit function theorem, we will prove that tight spherical 2-designs are
strong non rigid as Euclidean 2-designs.

Tight spherical 2-designs on $S^{n-1}$ are classified and
isometric to the regular simplex on $S^{n-1}$.
We can express the regular simplex with the following $n$ unit vectors $\boldsymbol  u_i=(u_{i,1},~\ldots,~u_{i,n})$, $i=1,~\ldots,~ n$, and
$\boldsymbol  u_{n+1}=\frac{1}{\sqrt{n}}(1,1,~\ldots, ~1)$ in $S^{n-1}$, where
\[
u_{i,j}=\left\{
\begin{array}{ll}b&\mbox{for $1\leq i\leq n$
and $j\neq i$},\\
a&\mbox{for $1\leq i\leq n$
and $j= i$},
\end{array}
\right.
\]

$a=-\frac{1+(n-1)\sqrt{n+1}}{n\sqrt{n}}$
and $b=\frac{-1+\sqrt{n+1}}{n\sqrt{n}}$.

\indent
Recall the explicit basis for $\mbox{Harm}_s(\RR^n).$
Let $\boldsymbol x=(x_1,\ldots x_n)\in \RR^n$.
Let
\begin{align*}\Phi_1\quad &
=\{f_k(\boldsymbol x)=x_k,\quad 1\leq k\leq n\},\\
\Phi_2\quad &=\{g_{1,k}(\boldsymbol x)=x_1x_k ,\quad 2\leq k\leq n\},\\
&\vdots\\
\Phi_i\quad &=\{g_{i-1,k}(\boldsymbol x)=x_{i-1}x_k ,\quad
i\leq k\leq n\},\\
&\vdots\\
\Phi_n\quad &=\{g_{n-1,n}(\boldsymbol x)=x_{n-1}x_n \},\\
\Phi_{n+1}&=\{h_k(\boldsymbol x)={x_1}^2-{x_k}^2,\quad
2\leq k \leq n \}.
\end{align*}
Then $\Phi_1$ is a basis of
$\mbox{Harm}_1(\RR^n)$ and
$\cup_{i=2}^{n+1}\Phi_i$
is a basis of $\mbox{Harm}_2(\RR^n)$. Note that
$\dim (\mbox{Harm}_1(\RR^n))+\dim (\mbox{Harm}_2(\RR^n))=\frac{n^2+3n-2}{2}$.

Let $\Phi=\cup_{i=1}^{n+1}\Phi_i$ and $I=\{x_{i,1},~\ldots,~x_{i,n},1\leq i\leq n+1\}\cup\{w_1,\ldots,w_{n+1}\}$, i.e., be the set of $(n+1)^2$ variables.

Theorem \ref{TheoN-S} implies that an $(n+1)$-point set $\{
\boldsymbol v_i=(v_{i,1},~\ldots,~v_{i,n}), 1\leq i\leq n+1\}$
is a Euclidean design if and only if
$\{v_{i,1},~\ldots,~v_{i,n},1\leq i\leq n+1\}\cup\{w_1,\ldots,w_{n+1}\}$ is a
solution of the following system of $\frac{n^2+3n-2}{2}$ equations in $(n+1)^2$ variables in $I$.
$$
\sum_{\lambda=1}^{n+1}w_\lambda \varphi(x_{\lambda,1},~\ldots,~x_{\lambda,n})=0,\quad \mbox{for all}\ \varphi\in \Phi.$$
Let $I_i=\{x_{j,i}\ |\ i\leq j\leq n\}$, for $i=1, 2,\ldots, n-1$ and $I_n=\{x_{1,n}, x_{2,n},\ldots,x_{n,n}\}$.
Then $\sum_{i=1}^n|I_i|=n+(n-1)+\cdots+2+n=\frac{n^2+3n-2}{2}$ holds.
Let $\boldsymbol x_{\lambda}=(x_{\lambda,1}, \ldots, x_{\lambda,n})$.
In the following we will prove that
the Jacobian $J$ of our system of
$\frac{n^2+3n-2}{2}$ equations in $(n+1)^2$
variables is of the full rank at $\boldsymbol x=\boldsymbol u, \ 1 \leq \lambda \leq n+1$ and $w_1=\cdots=w_{n+1}=1$. Actually we prove
that the sub matrix
$$J'=\left(\frac{\partial}{\partial x}
\left(\sum_{\lambda=1}^{n+1}w_\lambda \varphi(x_{\lambda,1},~\ldots,~x_{\lambda,n})\right)\right)_{\varphi\in \Phi, x\in \cup_{j=1}^{n} I_j}$$
is a regular matrix of size $\frac{n^2+3n-2}{2}$. In this case $I'(=I\backslash\cup_{j=1}^nI_j)=\{x_{i,j}\ |\ 1\leq i<j\leq n-1\}\cup\{x_{n+1,j}\ |\ 1\leq j\leq n\}\cup \{w_i\ |\ 1\leq i\leq n+1 \}$. If we prove this, for any $w_1,\ldots,w_{n+1}$ and $x_{i,j}$, $1\leq i<j\leq n-1$, in
a small neighborhood of
$w_1=1,\ldots,w_{n+1}=1$, $x_{i,j}=b$, $1\leq i<j\leq n$, $x_{n+1,j}=\frac{1}{\sqrt{n}}$, $1\leq j\leq n$,  we obtain a tight Euclidean
$2$-design in $\mathbb R^n$. Since, by Proposition \ref{Prop2.4}, a tight Euclidean $2$-design must have a constant weight on each sphere, this implies the existence of many non-isomorphic tight Euclidean $2$-designs in $\RR^n$ supported  by $p=2,3, \ldots,n+1$ concentric spheres, respectively.\\

\noindent
{\bf Proof }\\
In the following we use same symbol
$J$, $J'$ and $J_j$ for the same matrices
evaluated at $\boldsymbol x_{\lambda}=\boldsymbol u_{\lambda}$
and $w_1=\cdots=w_{n+1}=1$.
Let
$$J_j=\left(\frac{\partial}{\partial x_{i,j}}
\left(\sum_{\lambda=1}^{n+1}w_\lambda \varphi(x_{\lambda,1},~\ldots,~x_{\lambda,n})\right)\right)_{\varphi\in \Phi, x_{i,j}\in I_j}$$
for $1\leq i\leq n$.
Then $J'=[J_1,J_2,\ldots,J_n].$
For example, If $n=4$, we have
$$J'=\left[\begin{array}{ccccccccccccc}
1&1&1&1&0&0&0&0&0&0&0&0&0\\
0&0&0&0&1&1&1&0&0&0&0&0&0\\
0&0&0&0&0&0&0&1&1&0&0&0&0\\
0&0&0&0&0&0&0&0&0&1&1&1&1\\
b&a&b&b&b&b&b&0&0&0&0&0&0\\
b&b&a&b&0&0&0&b&b&0&0&0&0\\
b&b&b&a&0&0&0&0&0&a&b&b&b\\
0&0&0&0&b&a&b&b&b&0&0&0&0\\
0&0&0&0&b&b&a&0&0&b&a&b&b\\
0&0&0&0&0&0&0&b&a&b&b&a&b\\
2a&2b&2b&2b&-2a&-2b&-2b&0&0&0&0&0&0\\
2a&2b&2b&2b&0&0&0&-2a&-2b&0&0&0&0\\
2a&2b&2b&2b&0&0&0&0&0&-2b&-2b&-2b&-2a
\end{array}
\right]$$
In general, each $J_j$ are of the following shape.
Let
$$J_j=\left[\begin{array}{c}
A_{1,j}\\
\vdots\\
A_{n+1,j}
\end{array}\right],$$
where
$A_{i,j}$ ($1\leq j\leq n-1$, $1\leq i\leq n$) is an $(n-i+1)\times (n-j+1)$,
$A_{n+1,j}$ ($1\leq j\leq n-1$) is an $(n-1)\times (n-j+1)$ matrix, $A_{i,n}$
($1\leq i\leq n$) is an  $(n-i+1)\times n$ matrix and $A_{n+1,n}$ is an $(n-1)\times n$ matrix.
Each $A_{i,j}$ ($1\leq i\leq n+1$, $1\leq j\leq n$) is
defined by the following way:\\
\begin{itemize}
\item The $j$-th row vector of $A_{1,j}$, $(1\leq j\leq n)$, is $(1,1,\ldots,1)$ and all the other rows are 0.

\item The $(j-i+1)$-th row vector of
$A_{i,j}$ $(2\leq i\leq j\leq n-1)$ is
$(b,b,\ldots,b)$ and all the other vectors are zero.
\item $A_{j+1,j}$ $(1\leq j\leq n-1)$ is a
matrix whose $(k,k+1)$-th entry is $a$
for $k=1,\ldots, n-j$ and all the other
entries are $b$.
\item $A_{i,j}$,
$(1\leq j\leq n-2,\ j+2\leq i\leq n)$,
is a zero matrix.
\item Every row vector of $A_{n+1,1}$ is $(2a,2b,2b,\ldots,2b)$.
\item The $(j-1)$-th row vector of $A_{n+1,j}$
$(2\leq j\leq n-1)$ is $(-2a,-2b,-2b,\ldots,-2b)$ and all the other row vectors are 0.
\item The $(n-i+1)$-th
row vector of $A_{i,n}$ $(2\leq i\leq n)$ is
$(b,\ldots,b,a,b,\ldots,b)$, where $a$ is the $(i-1)$-th entry, and all the other row vectors are zero.
\item The $(n-1)$-th row vector of $A_{n+1,n}$
is $(-2b,-2b,\ldots,-2a)$ and all the other row vectors are zero.
\end{itemize}
Then it is easy to see that the rank of $J_j$
equals $n-j+1$ for $j=1,\ldots, n-1$ and
rank of $J_n$ is $n$. Let $W_j$ be the subspace of $\mathbb R^{\frac{n^2+3n-2}{2}}$ spanned by
the column vectors of $J_j$, then it is easy to
see that $W_j\cap W_l=\{0\}$ for any $j\not=l$.
Hence the rank of  $J'$ equals
$\frac{n^2+3n-2}{2}$ and $J$ has the full rank.\qed\\

\noindent
{\bf (2) Tight spherical 3-designs on $S^{1}.$} \\
It is known that every tight spherical $(2e+1)$-design is antipodal.  Furthermore, since $\dim(\mbox{Pol}_1(S^1))=\dim(\mbox{Pol}_1(\RR^2))$, every tight spherical $3$-design is an antipodal tight Euclidean $3$-design.

Recently, Bajnok \cite{Baj} constructed antipodal tight Euclidean $3$-designs in $\RR^2$ supported by $p=1,2$ concentric spheres as follow.  For $1 \leq p \leq 2,$ set $m=6-2p$ and
\[
X=\left\{b_{kj}=\left(r_k \cos\left(\frac{2j+k}{m}\pi\right),r_k \sin\left(\frac{2j+k}{m}\pi \right)\right),~1 \leq j \leq m,~1 \leq k \leq p \right\}.
\]
The weight function is given by $\displaystyle w(b_{kj})=\frac{1}{r_k^2},$ for $k=1,2$.

Later, Bannai \cite{Et} gave the complete classification of Euclidean designs of this type in $\RR^n:$
\[
X=\left\{\pm r_i \boldsymbol e_i, 1\leq i\leq n \right \}
\quad \mbox{and}
\quad
w(\pm r_i \boldsymbol e_i)=\frac{1}{n{r_i}^2}\quad
\mbox{for $1\leq i\leq n$,}\]
where $r_1,\ldots, r_n$ are any
positive real numbers.

From those results, we conclude that tight spherical $3$-designs are strongly non-rigid as Euclidean designs.

We can also show the strong non-rigidity of a tight spherical $3$-design $X=\{\pm(1,0),
\pm(0,1) \}$ on $S^1$ using the implicit function theorem.
Let $\boldsymbol u_1=(1,0), \boldsymbol u_3=(0,1)$,
$\boldsymbol u_2=-\boldsymbol u_1$, and $\boldsymbol u_4=-\boldsymbol u_3$.
As for the harmonic polynimials in $\mbox{Pol}(\RR^2)$ we use the following notation. Let $\boldsymbol x=(x_1,x_2)\in \RR^2$. We define
\begin{align*}&
\varphi_1(\boldsymbol x)=x_1,\quad \varphi_2(\boldsymbol x)=x_2,\quad \varphi_3(\boldsymbol x)=x_1x_2, \quad \varphi_4(\boldsymbol x)={x_1}^2-{x_2}^2,\\
&\varphi_5(\boldsymbol x)=3{x_1}^2x_2-{x_2}^3,
\quad
\varphi_6(\boldsymbol x)={x_1}^3-3x_1{x_2}^2,\\
&\varphi_7(\boldsymbol x)={x_1}^3x_2-x_1{x_2}^3,
\quad \varphi_8(\boldsymbol x)={x_1}^4-6{x_1}^2{x_2}^2+{x_2}^4,\\
&\varphi_9(\boldsymbol x)
={x_1}^5-10{x_1}^3{x_2}^2+5x_1{x_2}^4,\quad
\varphi_{10}(\boldsymbol x)
=5{x_1}^4{x_2}-10{x_1}^2{x_2}^3+{x_2}^5.
\end{align*}
Then we have
\begin{align*}
&\mbox{Harm}_1(\RR^n)=\langle\varphi_1,\varphi_2\rangle,\quad
\mbox{Harm}_2(\RR^n)=\langle\varphi_3,\varphi_4\rangle,\quad \mbox{Harm}_3(\RR^n)=\langle\varphi_5,\varphi_6\rangle,\\
&\mbox{Harm}_4(\RR^n)=\langle\varphi_7,\varphi_8\rangle,\quad
\mbox{Harm}_5(\RR^n)=\langle\varphi_9,\varphi_{10}\rangle.
\end{align*}
Let $\boldsymbol x_\lambda=(x_{\lambda,1},x_{\lambda ,2})
\in \RR^n$ for
$\lambda=1,2,3,4$.
Let us consider the following 8 polynomial functions in 12 variables $\{x_{\lambda,1},
x_{\lambda,2},
w_\lambda, \ 1\leq \lambda\leq 4\}$:
\begin{align*}
&f_i=\sum_{\lambda=1}^4w_\lambda
\varphi_i(\boldsymbol x_\lambda),\quad 1\leq i\leq 6,\\
&f_{6+i}=\sum_{\lambda=1}^4w_\lambda\|\boldsymbol x_\lambda\|^2
\varphi_i(\boldsymbol x_\lambda),\quad 1
\leq i\leq 2.
\end{align*}
Then $\{\boldsymbol x_\lambda,i=1,\ldots,4 \}$ is a Euclidean 3-design with weight $w(\boldsymbol x_\lambda)= w_\lambda$ if and only if $f_i=0$ holds for $i=1,\ldots, 8$.
For our purpose, we construct the Jacobian $J^\prime$ with rows and columns indexed by the polynomial of even degree, i.e., $f_3$ and $f_4$, and the antipodal half of $X$, i.e.,
$x_{1,1},x_{1,2},x_{3,1},x_{3,2},w_1,w_3,$ respectively.  If we take $I^\prime=\{x_{1,1},x_{1,2},x_{3,1},x_{3,2}\}$, the Jacobian $J^\prime$
at the given solution,
$\boldsymbol x_\lambda=\boldsymbol u_\lambda, w_\lambda=1,1\leq\lambda\leq 4$,
is of full rank, i.e., rank({$J^\prime$})=2, at $X(r)$.  It means that we can move the points $\{(1,0), (0,1)\}$ (and simultaneously their antipodal pair $\{(-1,0),(0,-1)\}$) slightly and freely such that they sit on two different concentric spheres.\\

\noindent
{\bf (3) Tight 4-designs in $\RR^2.$}\\
We have known that the tight Euclidean $4$-designs $X(r)$ in $\RR^2$ constructed by Bannai and Bannai \cite{B-B-2} are non-rigid (see Example \ref{Ex3.4}).
Let $\boldsymbol u_1=(1,0),\
\boldsymbol u_2=(-\frac{1}{2},\frac{\sqrt{3}}{2}),\
\boldsymbol u_3=(-\frac{1}{2},-\frac{\sqrt{3}}{2})$,
$\boldsymbol u_4=(-r,0),\
\boldsymbol u_5=(\frac{1}{2}r,\frac{\sqrt{3}}{2}r),\
\boldsymbol u_6=(\frac{1}{2}r,-\frac{\sqrt{3}}{2}r)$. Then $w(\boldsymbol u_\lambda)=1$ for $\lambda=1,2,3$, and
$w(\boldsymbol u_\lambda)=\frac{1}{r^3}$ for $\lambda=4,5,6$.

Let $\boldsymbol x_\lambda=(x_{\lambda,1},x_{\lambda,2}), 1\leq \lambda \leq 6$.
We check the strong non-rigidity of the designs using the implicit function theorem with the following 12 polynomial functions in 18 variables $\{x_{\lambda,1},
x_{\lambda,2},
w_\lambda, \ 1\leq \lambda\leq 6\}$:
\begin{align*}&f_i=\sum_{\lambda=1}^6w_\lambda
\varphi_i(\boldsymbol x_\lambda),\quad 1\leq i\leq 8\\
&f_{8+i}=\sum_{\lambda=1}^6w_\lambda\|\boldsymbol x_\lambda\|^2
\varphi_i(\boldsymbol x_\lambda),\quad 1
\leq i\leq 4.
\end{align*}
Then $\{\boldsymbol x_\lambda,\ 1\leq \lambda\leq 6\}$ is a Euclidean 4-design
with weight $w(\boldsymbol x_\lambda)=w_\lambda$ if and only if
$f_i=0$ holds for $i=1,\ldots,12$.

If we  take, say, $I^{\prime}=\{w_1,w_2,w_4,w_5\}$  we get that the Jacobian $J^{\prime}$ is of full rank at $X(r)$ (i.e., rank$(J^{\prime})=12$).  Therefore, we can move $w_1,w_2,w_4,w_5$ slightly and freely in such a way that some two of them have the same value, or even all $w_i$'s are distinct to each other.  This implies the existence of many non-isomorphic tight 4-designs with $p=3 \mbox{ and }4$ in $\RR^2$.\\

\noindent
{\bf (4) Antipodal tight 5-designs in $\RR^2.$} \\
We have also known that the antipodal tight Euclidean $5$-designs $X(r)$ in $\RR^2$ constructed by Bannai \cite{Et} are non-rigid (see Example \ref{Ex3.5}).  Here we define
$\boldsymbol u_1=(1,0),\
\boldsymbol u_2=-\boldsymbol u_1,\
\boldsymbol u_3=(0,1),\
\boldsymbol u_4=-\boldsymbol u_3,\
\boldsymbol u_5
=(\frac{r}{\sqrt{2}},\frac{r}{\sqrt{2}}),\
\boldsymbol u_6=-\boldsymbol u_5,\
\boldsymbol u_7
=(\frac{r}{\sqrt{2}},-\frac{r}{\sqrt{2}}),\
\boldsymbol u_8
=-\boldsymbol u_7$.
Then $w(\boldsymbol u_\lambda)=1$ for
$\lambda=1,\ldots, 4$ and $w(\boldsymbol u_\lambda)=\frac{1}{r^4}$ for
$\lambda=5,\ldots, 8$.
Again we use the implicit function theorem.
Let $\boldsymbol x_\lambda
=(x_{\lambda,1},x_{\lambda,2})$ for $\lambda=1,\ldots, 8$. We use the
following 18 polynomial functions in 24 variables $\{x_{\lambda,1},x_{\lambda,2},w_\lambda, 1\leq \lambda\leq 8\}$ to check the strong non-rigidity of the designs.
\begin{align*}
f_i &=\sum_{\lambda=1}^8
w_\lambda\varphi_i(
\boldsymbol x_\lambda), \mbox{for}\ i=1,\ldots,10,
\\
f_{10+i} &=\sum_{\lambda=1}^8
w_\lambda\|\boldsymbol x_\lambda\|^2\varphi_i(\boldsymbol x_\lambda),\ \mbox{for}\ i=1,\ldots 6,\\
f_{16+i} &=\sum_{\lambda=1}^8
w_\lambda\|\boldsymbol x_\lambda\|^4\varphi_i(\boldsymbol x_\lambda),\ \mbox{for}\ i=1,\ 2.
\end{align*}
Then $\{\boldsymbol x_\lambda,1\leq \lambda\leq 8\}$ is a Euclidean 5-design with weight $w(\boldsymbol x_\lambda)=w_\lambda$
if and only if $f_i=0$ holds for $i=1,\ldots, 18$.
Analogue to the case of $3$-designs mentioned above, after neglecting the function $f_k$ of odd degree, we can take, for instance, $I^\prime=\{x_{1,1},x_{1,2},x_{3,1},x_{3,2}\}$ to get the Jacobian $J^\prime$ of full rank, rank$(J^\prime)=6$, at $X(r)$.  It means that we can move the points $\{(1,0),(0,1)\}$ (and simultaneously their antipodal pair $\{(-1,0),(0,-1)\}$) slightly and freely such that they sit on two different concentric spheres.  This implies that the antipodal $5$-designs in $\RR^2$ is strongly non-rigid.  Moreover, the Jacobian is again of full rank if we take $I^\prime=\{w_1,w_3,w_5,w_7\}$.  By the same reason with the $4$-designs case above, this implies the existence of many non-isomorphic antipodal tight Euclidean $5$-designs in $\RR^2$ supported by $3\mbox{ and }4$ concentric spheres, respectively.
\\

\begin{remark}
We have also investigated the possibility of existence of non-antipodal Euclidean $3$- and $5$-design in $\RR^2$ with cardinalities 4 and 8 respectively.  Using the implicit function theorem, we consider the Jacobian whose rows are indexed by eight and eighteen functions mentioned above for the case of $3$- and $5$-design, respectively. In the case of $5$-design, for getting the Jacobian $J^\prime$ of full rank, i.e., rank$(J^\prime)=18$, at $X(r)$ there exist several choices of $I^\prime, |I^\prime|=4,$ consisting of two points sitting on the same sphere.  But none of them are antipodal pair.  In the case of $3$-design, the result is "worst": there is no any choice of such an $I^\prime$ which gives the corresponding Jacobian of full rank.  This implies that the implicit function theorem doesn't work to show the existence of non-antipodal Euclidean $3$- and $5$-design in $\RR^2$, respectively.
\end{remark}

\section{Tight Euclidean $2$-designs in $\RR^n$}\label{Sec6}
In the previous section we have shown that tight spherical $2$-designs in $\RR^n$ are strongly non-rigid and hence there exist infinitely many (non-isomorphic) tight Euclidean $2$-designs in $\RR^n$ supported by $2,3,\ldots,n+1$ concentric spheres, respectively.  The aim of this section is to give the complete classification of tight Euclidean $2$-designs in $\RR^n$.

By Proposition \ref{Prop2.6} (2) and the fact that in $\RR^n$ the Gegenbauer polynomial of degree $1$ satisfies
\[
Q_1(y)=ny,
\]
we obtain
\[
\langle \boldsymbol u,
\boldsymbol v \rangle =-\frac{a_1}{na_0},~\mbox{for any distinct vectors }\boldsymbol u,\boldsymbol v \in X.
\]
Therefore every tight Euclidean $2$-design $X$ is a 1-inner product set with negative inner product value $\displaystyle -\frac{a_1}{na_0}.$  In general, a subset $X \seq \RR^n$ is called $e$-inner product set if
\[
|\{\langle \boldsymbol x,\boldsymbol y \rangle,~\boldsymbol x,\boldsymbol y \in X,~\boldsymbol x \neq \boldsymbol y \}|=e
\]
holds.  The cardinality of $e$-inner product set in $\RR^n$ is known to be bounded from above by $\displaystyle \left( n+e \atop e \right)$ (see \cite{D-F}).  In particular, a $1$-inner product set is bounded above by $n+1$ which is attained by regular simplices which is also tight spherical $2$-designs and tight Euclidean $2$-designs at the same time.

For any positive real numbers $R_1,R_2, \ldots,R_n$, we define a function $f_k$ of $k$ variables $R_1,R_2,\ldots,R_k$ by the recurrence relation as follows:
\begin{equation}\label{Eq4.1}
\left\{
\begin{array}{rcl}
f_0&=&1\\
f_1&=&R_1,\\
\displaystyle f_k&=&f_{k-1}(R_k+1)-\prod_{i=1}^{k-1}(R_i+1), \mbox{ for } 2 \leq k \leq n.
\end{array}
\right.
\end{equation}

Then we have the following theorem.

\begin{theorem}\label{Theo4.1}
Let $X=\{\boldsymbol x_k,~1 \leq k \leq n+1 \}$ be an $(n+1)$-subset in $\RR^n.$  Let also $R_k=\|\boldsymbol x_k\|^2,$ for $1 \leq k \leq n+1$.  If $X$ is a $1$-inner product set satisfying
\begin{equation}\label{Eq4.2}
\langle \boldsymbol x, \boldsymbol y \rangle = -1, \mbox{ for any distinct }\boldsymbol x,\ \boldsymbol y \in X,
\end{equation}
then the following three conditions hold:
\begin{itemize}
\item[(1)] $f_k > 0, \mbox{ for } 1 \leq k \leq n,$
\item[(2)] $f_n<\prod_{i=1}^n
(R_i+1)$,
\item[(3)] $\displaystyle R_{n+1}
=\frac{\prod_{i=1}^n(R_i+1)}{f_n}-1$.
\end{itemize}
Conversely, if the conditions (1), (2), and (3) hold, then there exists $1$-inner product set $X=\{\boldsymbol x_k,~1 \leq k \leq n+1\} \seq \RR^n$ satisfying $\|\boldsymbol x_k\|^2=R_k$ for
$k=1,\ldots, n+1$ and the condition (\ref{Eq4.2}).
\end{theorem}

\begin{proof}
Let $X=\{\boldsymbol x_k=(x_{k,1},x_{k,2},\ldots, x_{k,n}),~1 \leq k \leq n+1\} \seq \RR^n$ be a $1$-inner product set satisfying the condition (\ref{Eq4.2}).  Then up to the action of $O(n)$ we may assume that $x_{k,l}=0$, for $1 \leq k < l \leq n$ and $x_{k,k} \geq 0,$ for $1 \leq k \leq n$.  Let us first prove Claim $1$ below.\\

{\bf Claim 1:}  For $1 \leq k \leq n$, we have $x_{k,k} > 0,$ and $x_{j,k-1}=x_{j+1,k-1}= \ldots =x_{n+1,k-1}=b_{k-1},$ with some real number $b_{k-1} < 0,$ for $k \leq j \leq n+1$.

\emph{Proof of Claim 1.}  We use a mathematical induction on $k$.  Since $R_1=\|\boldsymbol x_1\|^2$ and $\boldsymbol x_1=(x_{1,1},0,\ldots,0),$ we have $x_{1,1}=\sqrt{R_1} > 0$.  Moreover, by (\ref{Eq4.2}) we have $\displaystyle x_{j,1}=-\frac{1}{\sqrt{R_1}}$, for $2 \leq j \leq n+1$.
Hence $\displaystyle b_1=-\frac{1}{\sqrt{R_1}} < 0.$
Now we assume that
$x_{i,i}>0$ and $x_{i+1,i}=x_{i+2,i}=\cdots=x_{n+1,i}=b_i$ holds for any $i\leq k-1$ with a real number $b_i<0$.
Then $\langle \boldsymbol x_l,\boldsymbol x_{k}\rangle =-1$ implies $\sum_{i=1}^{k-1}b_i^2+x_{k,k}x_{l,k}=-1$
 for any $l\geq k+1$. Therefore $x_{k,k}$ must be positive
and $x_{l,k}=-\frac{1+\sum_{i=1}^{k-1}b_i^2}{x_{k,k}}<0$
holds for any $l=k+1,\ldots, n+1$. Then
$b_k=-\frac{1+\sum_{i=1}^{k-1}b_i^2}{x_{k,k}}$ and the Claim 1 is true for $k$. This completes the  proof of
Claim 1.\\

Next we will express $b_1,\ldots,b_n$ and $x_{1,1},
\ldots,x_{n,n}$ interms of $R_1,\ldots,R_n$.
Since $\sum_{i=1}^{k-1}{b_i}^2+{x_{k,k}}^2=R_k$,
we have ${b_1}^2=\frac{1}{R_1}=\frac{1}{f_1}$ and  ${x_{2,2}}^2
=\frac{R_1R_2-1}{R_1}$. Hence we have $R_1R_2-1>0$. Let $f_1=R_1$ and $f_2=R_1R_2-1$.
Then we have $f_2=(R_2+1)f_1-(R_1+1)$.
This is (\ref{Eq4.1}) with $k=2$ and we also have
$x_{2,2}=\sqrt{\frac{f_2}{f_1}}$ and $1+{b_1}^2=\frac{R_1+1}{R_1}=\frac{R_1+1}{f_1}$. Then we have $b_2=-\frac{1+{b_1}^2}{x_{2,2}}
=-\frac{R_1+1}{\sqrt{f_1f_2}}$. Then
${b_1}^2+{b_2}^2+{x_{3,3}}^2=R_3$ implies
$${x_{3,3}}^2=R_3-{b_1}^2-{b_2}^2=\frac{R_1R_2R_3-R_1-R_2-R_3-2}{R_1R_2-1}.
$$
Hence $R_1R_2R_3-R_1-R_2-R_3-2>0$ holds.
We observe that
$$R_1R_2R_3-R_1-R_2-R_3-2=(R_3+1)(R_1R_2-1)
-(R_1+1)(R_2+1)
$$
holds.
Let $f_3=(R_3+1)(R_1R_2-1)
-(R_1+1)(R_2+1)$.
Then $f_3=(R_3+1)f_2-\prod_{i=1}^2(R_i+1)$.
This implies (\ref{Eq4.1}) with $k=3$, then we have
$x_{3,3}=\sqrt{\frac{f_3}{f_2}}$ and
$1+{b_1}^2+{b_2}^2=\frac{(R_1+1)(R_2+1)}{R_1R_2-1}=\frac{\prod_{i=1}^2(R_i+1)}{f_2}$.
Now we will prove the following claim.\\

{\bf Claim 2:}  Let $f_0, f_1,f_2,\ldots, f_n$ be the real numbers defined by (\ref{Eq4.1}) above. Then $f_k >0$, for $0 \leq k \leq n$ and the following hold for
$k=1,\ldots,n$:
\begin{itemize}
\item[(1)] $\displaystyle 1+\sum_{i=1}^k{b_i}^2=\frac{\prod_{i=1}^{k}(R_i+1)}{f_k},$
\item[(2)] $x_{k,k}=\displaystyle \sqrt{\frac{f_k}{f_{k-1}}}, \mbox{ for } 1 \leq k \leq n$,
\item[(3)] $\displaystyle b_k=-\frac{\prod_{i=1}^{k-1}(R_i+1)}{\sqrt{f_{k-1}f_k}}$.
\end{itemize}
\emph{Proof of Claim 2.}  We have already proved that Claim 2 holds for $k\leq 2$. We use induction on $k$. Assume $f_i>0$, $1+\sum_{l=1}^i{b_l}^2=\frac{\prod_{l=1}^i(R_l+1)}{f_i}$ and
$x_{i,i}=\sqrt{\frac{f_i}{f_{i-1}}}$ hold for any
$i=1,\ldots,k$.
Then
$$x_{k+1,k+1}^2=R_{k+1}
-\sum_{i=1}^k{b_i}^2=R_{k+1}+1-\frac{\prod_{i=1}^k(R_i+1)}{f_k}$$
holds. Therefore
$$R_{k+1}+1-\frac{\prod_{i=1}^k(R_i+1)}{f_k}>0$$
holds and this implies $f_{k+1}>0$ and $x_{k+1,k+1}=
\sqrt{\frac{f_{k+1}}{f_{k}}}$. Then
\begin{align*}
&1+\sum_{i=1}^{k+1}{b_i}^2
=\frac{\prod_{i=1}^k(R_i+1)}{f_k}+{b_{k+1}}^2
=\frac{\prod_{i=1}^k(R_i+1)}{f_k}+\left(\frac{1+\sum_{i=1}^k{b_i}^2}{x_{k+1,k+1}}\right)^2\\
&=\frac{\prod_{i=1}^k(R_i+1)}{f_k}+
\left(\frac{\prod_{i=1}^k(R_i+1)}{f_k}\right)^2\frac{f_k}{f_{k+1}}=\frac{\prod_{i=1}^k(R_i+1)}{f_k}
\left(1+\frac{\prod_{i=1}^k(R_i+1)}{f_{k+1}}\right)\\
&=\frac{\prod_{i=1}^{k+1}(R_i+1)}{f_{k+1}}.
\end{align*}
Finally, $b_{k+1}=-\frac{1+\sum_{i=1}^k{b_i}^2}{x_{k+1,k+1}}=- \frac{\prod_{i=1}^k(R_i+1)}{\sqrt{f_kf_{k+1}}}$. This completes the proof of Claim 2.\\

Since $1+R_{n+1}=1+\sum_{i=1}^n{b_i}^2
=\frac{\prod_{i=1}^n(R_i+1)}{f_n}$,
we have $R_{n+1}=\frac{\prod_{i=1}^n(R_i+1)}{f_n}-1>0$. Hence we obtain (2) and (3) of Theorem
\ref{Theo4.1}

Conversely, assume that the conditions (1), (2), and (3) of Theorem \ref{Theo4.1} hold, and let $R_1, R_2,\ldots,R_n$ be n positive real numbers satisfying these conditions.  Define positive real numbers $x_{1,1},\ldots, x_{n,n}$ using the equation (1) of Claim 2
and negative real numbers  $b_1,\ldots,b_n$ using the equation (2) of Claim 2. Here we define $b_1=-\frac{1}{\sqrt{R_1}}$. Then
define ($n+1$)-subset $X=\{
\boldsymbol x_k,~1 \leq k \leq n+1 \} \seq \RR^n$ by
\begin{align*}
\boldsymbol x_1 & =  (x_{1,1},~0,~0,~\ldots,~0),\\
\boldsymbol x_k & =  (b_1,~b_2,~b_3,~\ldots,~b_{k-1},~x_{k,k},~0,~\ldots,~0), \mbox{ for }2 \leq k \leq n; \mbox{ and }\\
\boldsymbol x_{n+1} & =  (b_1,~b_2,~b_3,~\ldots,~b_n).
\end{align*}
Then by direct inspection, it is easy to see that $R_k=\|\boldsymbol x_k\|^2,$ for $1 \leq k \leq n+1;$ and $\langle \boldsymbol x_k,\boldsymbol x_l \rangle=-1,$ for $k \neq l,$ i.e., $X$ satisfies the condition (\ref{Eq4.2}).
\end{proof}

\begin{corollary}
For any positive integer $p\leq n+1$, there always exists an $(n+1)$-point set $X \seq \RR^n$
satisfying the following conditions:
\begin{itemize}
\item[(1)] $\langle\boldsymbol x,
\boldsymbol y\rangle=-1$ for any distinct points $\boldsymbol x,\boldsymbol y\in X$.
\item[(2)] $|\{\|\boldsymbol x\|\ |\ \boldsymbol x\in X\}|=p$.
\end{itemize}
\end{corollary}

\begin{proof}
A regular simplex $X$ on the unit sphere $S^{n-1}$ is a 1-inner product set
with the inner product
$-\frac{1}{n}$. Hence $\{\sqrt{n} \boldsymbol x\ | \ \boldsymbol x\in X\}$ is a
1-inner product set with the inner product $-1$.
Let $R_1=\ldots=R_{n+1}=n$. Then the real numbers
defined by $f_0=1$, $f_i=(n-i)(n+1)^{i-1}$, for $i=1,\ldots,n$ satisfy the conditions (1), (2) and (3)
of Theorem \ref{Theo4.1}. Then any
$(R'_1,\ldots, R'_n)\in \RR^n$
in a small neighborhood of
$(n,\ldots,n)$
satisfies the conditions (1), (2) and (3)
of Theorem \ref{Theo4.1}. Hence we can make $|\{R'_1,\ldots, R'_n, R'_{n+1}\}|=p$ for any $1\leq p\leq n+1$
\end{proof}
In view of Proposition \ref{Prop2.6}, we have the theorem below.
\begin{theorem} Assume $|X|=n+1$. Then
$(X,w) \seq \RR^n$ is a Euclidean $2$-design if and only if $X$ is a weighted $1$-inner product set in $\RR^n$ of negative inner-product value.
\end{theorem}

\begin{proof}
Let $X \seq \RR^n$ be a Euclidean $2$-design with $|X|=n+1.$  Then $X$ is a tight Euclidean $2$-design and Proposition \ref{Prop2.6} implies
\[
\langle \boldsymbol u,\boldsymbol v \rangle = -\frac{a_1}{na_0} <0,~\mbox{for any distinct }\boldsymbol u,\boldsymbol v \in X,
\]
namely $X$ is a $1$-inner product set of negative inner-product value $\displaystyle -\frac{a_1}{n a_0}$ consisting of $n+1$ points.

Conversely,  let $X=\{\boldsymbol x_k,~1 \leq k \leq n+1\}$, with $R_k=\|\boldsymbol x_k\|^2,$ for $1 \leq k \leq n+1$ be a $1$-inner product set in $\RR^n$ of negative inner-product value.  Proposition \ref{Prop2.2}
implies that by scaling we may assume
\[
\langle \boldsymbol u,\boldsymbol v \rangle = -1,~\mbox{for any distinct }
\boldsymbol u,\boldsymbol v \in X.
\]
Let $\displaystyle w(\boldsymbol x)=\frac{1}{\|\boldsymbol x\|^2+1}$ be a weight function of $X$.  It is enough for us to show that the equations (\ref{Eq2.3}) and (\ref{Eq2.4}) of Proposition \ref{Prop2.6} hold for $e=1.$
By definition (as given in (\ref{Eq2.2})), we have
\[
a_0=\sum_{\boldsymbol x \in X}
w(\boldsymbol x)=\sum_{\boldsymbol x \in X}\frac{1}{\|\boldsymbol x\|^2+1}=\frac{1}{R_{n+1}+1} + \sum_{i=1}^n \frac{1}{R_i+1}.
\]

Moreover, by the recurrence relation (\ref{Eq4.1}), together with the conditions (1) and (2) of Theerem \ref{Theo4.1}, the last expression above is equal to the following

\begin{align*}
\frac{1}{R_{n+1}+1} + \sum_{i=1}^n \frac{1}{R_i+1} & =  \frac{f_n}{\prod_{i=1}^n(R_i+1)} + \sum_{i=1}^n \frac{1}{R_i+1}\\
& =  \frac{(R_n+1)f_{n-1}-\prod_{i=1}^{n-1}(R_i+1)}{\prod_{i=1}^n(R_i+1)} + \sum_{i=1}^n \frac{1}{R_i+1}\\
& =  \frac{f_{n-1}}{\prod_{i=1}^{n-1}(R_i+1)} + \sum_{i=1}^{n-1} \frac{1}{R_i+1}\\
& \vdots & \\
& =  \frac{f_2}{(R_2+1)(R_1+1)}+\frac{1}{R_2+1}+\frac{1}{R_1+1}\\
& =  1.
\end{align*}

Hence we have $a_0=1.$  Also by definition as given in (\ref{Eq2.2}), we have
\[
\sum_{\boldsymbol x \in X}
w(\boldsymbol x)
\|\boldsymbol x\|^2=a_1.
\]
Since the weight function $\displaystyle w(\boldsymbol x)=\frac{1}{1+
\|\boldsymbol x\|^2},~\mbox{for } \boldsymbol x \in X,$ we have
$w(\boldsymbol x) + w(\boldsymbol x) \|\boldsymbol x\|^2 =1, ~\mbox{for }\boldsymbol x \in X$.
This implies
\[
\sum_{\boldsymbol x \in X}w(\boldsymbol x) \|\boldsymbol x\|^2 +\sum_{\boldsymbol x \in X}w(\boldsymbol x) = n+1.
\]
Since $a_0=1,$ we obtain $a_1=n.$  Bearing in mind that $Q_1(\alpha)=n \alpha$, we have
\[
\frac{\|\boldsymbol x\|^2}{a_1}Q_1(1) + \frac{1}{a_0}=\|\boldsymbol x\|^2 + 1 = \frac{1}{w(\boldsymbol x)},
\]

and

\[
\frac{\|\boldsymbol x\| \|\boldsymbol y\|}{a_1}Q_1\left(\frac{\langle
\boldsymbol x,\boldsymbol y \rangle}
{\|\boldsymbol x\| \|\boldsymbol y\|} \right) + \frac{1}{a_0} = \frac{\|\boldsymbol x\| \|\boldsymbol y\|}{n}Q_1\left(\frac{-1}{\|\boldsymbol x\|
\|\boldsymbol y\|} \right) + 1 = 0.
\]

Hence Proposition \ref{Prop2.6} implies that $X$ is a tight Euclidean $2$-design with weight function $\displaystyle w(\boldsymbol x)=\frac{1}{1+\|\boldsymbol x\|^2},~\mbox{for }\boldsymbol x \in X$.
\end{proof}

\section{Concluding Remarks}
\begin{itemize}
\item[(1)] Neumaier and Seidel and also Delsarte and Seidel
conjectured that the only tight Euclidean $2e$-designs in
$\RR^n$ are regular simplices (See \cite[Conjecture 3.4]{N-S}
and \cite[pp. 225]{D-S}).  Recently, Bannai and Bannai
\cite{B-B-2} has disproved this conjecture providing the
example of Euclidean tight $4$-designs in $\RR^2$ supported
by two concentric spheres, i.e., which are not regular simplices.
However, constructing a tight Euclidean design is not so easy
in general.  In this paper we introduce a new notion of a
strong non-rigidity of Euclidean $t$-designs. Then, disprove
the conjecture by investigating the strong non-rigidity of
the designs.

\item[(2)] Regarding the existence of tight Euclidean designs,
we believe in the following conjecture:
\begin{conjecture}If a tight Euclidean $2e$-design or an
antipodal tight Euclidean $(2e+1)$-design supported by more
than $\displaystyle \left[\frac{e+\varepsilon_S}{2}\right]+1$
concentric spheres exists, then there exist infinitely many
tight Euclidean $2e$-designs or antipodal tight Euclidean
$(2e+1)$-designs, respectively.
\end{conjecture}
\end{itemize}

\end{document}